\documentclass{article}       
\usepackage{geometry,t1enc,amsfonts,latexsym,amsmath,amsthm,amssymb,graphics,mathtools}
\usepackage[utf8]{inputenc}
\usepackage{authblk}

\newtheorem{theorem}{Theorem}
\newtheorem{proposition}{Proposition}
\newtheorem{lemma}{Lemma}

\def\r{{\mathcal R}}
\def\I{{\mathcal I}}
\def\N{{\mathbb N}}
\def\m{{\mathbb M}}

\def\R{{\mathbb R}}

\newcommand{\argmin}[1]{\underset{#1}{\mathrm{argmin~}}}
\newcommand{\argmax}[1]{\underset{#1}{\mathrm{argmax~}}}

\DeclareMathOperator{\Rank}{\mathsf{rank}}

\numberwithin{equation}{section}

\mathtoolsset{showonlyrefs}  
\begin{document}

\title{Convex envelopes for fixed rank approximation}
\author{Fredrik Andersson \thanks{fa@maths.lth.se}}
\author{Marcus Carlsson \thanks{mc@maths.lth.se}}
\author{Carl Olsson \thanks{calle@maths.lth.se}}
\affil{Centre for Mathematical Sciences, Lund University \protect \\ Box 118, SE-22100, Lund,  Sweden}
\date{}

\maketitle

\begin{abstract}
A convex envelope for the problem of finding the best approximation to a given matrix with a prescribed rank is constructed. This convex envelope allows the usage of traditional optimization techniques when additional constraints are added to the finite rank approximation problem. Expression for the dependence of the convex envelope on the singular values of the given matrix is derived and global minimization properties are derived. The corresponding proximity operator is also studied.
\end{abstract}

\section{Introduction}
Let $\m_{m,n}$ denote the Hilbert space of complex $m\times n$-matrices equipped with the Frobenius (Hilbert-Schmidt) norm. The Eckart--Young--Schmidt theorem \cite{eckart1936approximation,schmidt1908theorie} provides a solution to the classical problem of approximating a matrix by another matrix with a prescribed rank, i.e.,
\begin{equation} \label{fixrank_standardproblem}
\begin{aligned}
&\min \| A- F\|^2 \\
&\text{subject to } \Rank A \leq K,
\end{aligned}
\end{equation}
 by means of a singular value decomposition of $F$ and keeping only the $K$ largest singular vectors. However, if additional constraints are added then there will typically not be an explicit expression for the best approximation.
 
  Let $g(A)=0$ describe the additional constraints (for instance imposing a certain matrix structure on $A$), and consider
\begin{equation} \label{fixrank_constrained}
\begin{aligned}
&\min \| A- F\|^2 \\
&\text{subject. to } \Rank A \leq K, \quad g(A)=0.
\end{aligned}
\end{equation}
The problem \eqref{fixrank_constrained} can be reformulated as
\begin{equation} \label{fixrank_constrained_r}
\begin{aligned}
&\min \I(A) = \r_K(A) +\|A-F\|^2 \\
&\text{subject. to } g(A)=0.
\end{aligned}
\end{equation}
where
$$
\r_K(A)=\left\{\begin{array}{cc}
                     0 & \Rank A\leq K,\\
                     \infty & \text{ else.}
                   \end{array}
\right.
$$
For instance, if $g$ describes the condition that $A$ is a Hankel matrix and $F$ is the Hankel matrix generated by some vector $f$, then the minimization problem above is related to that of approximating $f$ by $K$ exponential functions \cite{kronecker1881theorie}. This particular case of \eqref{fixrank_constrained_r} was for instance studied in \cite{andersson2014new}.

Standard (e.g. gradient based) optimization techniques do no work on \eqref{fixrank_constrained_r} due to the highly discontinuous  behavior of the rank function. A popular approach is to relax the optimization problem by replacing the rank constraint with a nuclear norm penalty, i.e. to consider the problem
\begin{equation} \label{fixrank_nuclear}
\begin{aligned}
&\mu_K \|A\|_\ast +\|A-F\|^2 \\
&\text{subject. to } g(A)=0.
\end{aligned}
\end{equation}
where $\|A\|_\ast = \sum_j \sigma_j(A)$, where the parameter $\mu_K$ is varied until the desired rank $K$ is obtained. 

In contrast to $\r_K(A)$ the nuclear norm $\|A\|_\ast$ is a convex function, and hence \eqref{fixrank_nuclear} is much easier to solve than \eqref{fixrank_constrained_r}. In fact, the nuclear norm is the convex envelope of the rank function (acting on matrices with operator norm $\le 1$)  \cite{fazel2002matrix} which motivates the replacement of $\r_K(A)$ with $\mu_K \|A\|_\ast$ (for a suitable choice of $\mu_K$).

However, the solutions obtained by solving this relaxed problem are often not good enough as approximations of the original problem. In fact the relaxation with replacing $\r_K(A)$ with $\mu_K \|A\|_\ast$ is not optimal even though  the nuclear norm is the convex envelope of the rank function. This is because the contribution of the (convex) misfit term $\|A-F\|^2$  is not used. In \cite{larsson2015convex,larsson2014rank} it was suggested to incorporate the misfit term and work with the convex envelopes of
\begin{equation}\label{lowrank}
\mu \Rank(A) + \|A-F\|^2,
\end{equation}
and
\begin{equation}\label{fixrank}
\r_K(A) + \|A-F\|^2,
\end{equation}
respectively for the problem of low-rank and fixed rank approximations. The superior performance of using this relaxation approach in comparison to the nuclear norm approach was also verified by several examples in \cite{larsson2015convex,larsson2014rank}. For the convex envelope of \eqref{lowrank} it turns out that there are simple explicit formulas acting on each of the singular values of $F$ individually. This is not the case for the convex envelope of \eqref{fixrank}. Nevertheless, in \cite{larsson2015convex,larsson2014rank} an efficient optimization algorithm is presented that acts only on the singular values of $F$.

In this paper we present explicit expressions for the convex envelope of \eqref{fixrank} in terms of the singular values $(\alpha_j)_{j=1}^{\min(m,n)}$ of $A$, as well a detailed information about global minimizers. More precisely, in Theorem \ref{t1} we show that the convex envelope of \eqref{fixrank} is given by 
\begin{align}\label{t1prel} 
\frac{1}{k_*}\bigl(\sum_{j>K-k_*}\alpha_j \bigr)^2 -\sum_{j>K-k_*}\alpha_j^2+\|A-F\|^2.
\end{align} 
where $k_*$ is a particular value between 1 and $K$. To determine this value one uses Lemma \ref{l1} (Section \ref{sec1}). The second main result of this note is Theorem \ref{thm_multiplicity}, where the global minimizers of \eqref{t1prel} are found. In case the $K:$th singular value of $F$ (denoted $\phi_K$) has multiplicity one, then the minimizer of \eqref{t1prel} is unique and coincides with that of \eqref{fixrank}, given by the Eckart-Young-Schmidt theorem. If $\phi_K$ has multiplicity $M$ and is constant between sub-indices $J\leq K\leq L$, it turns out that the singular values $\alpha_j$ of global minimizers $A$, in the range $J\leq j\leq L$ lie on a certain simplex in $\R^M$. We refer to Section \ref{secglobal} and \eqref{15} for further details.

In Section \ref{prox} we investigate the propertties of the  proximal operator $$A\mapsto\argmin{A} \r_K(A) + (1+\rho)\|A-F\|^2, \quad \rho>0.$$ In particular we show that it is a contraction with respect to the Frobenius norm and show that the proximal operator coincides with the solution of \eqref{fixrank_standardproblem} whenever $F$ has a sufficient gap between the $K$:th and $K+1$:th singular value (see \eqref{n}).

\section{Fenchel conjugates and the convex envelope}\label{sec1}
The Fenchel conjugate, also called the Legendre transform \cite[Section 26]{rockafellar2015convex}, of a function $f$ is defined by
$$
f^\ast (B) = \argmax{A}\langle A,B\rangle-f(A).
$$
Note that for any function $f:\m\rightarrow\R$ that only depends on the singular values, we have that the maximum of $\langle A,B\rangle-f(A)$ with respect to $A$ is achieved for a matrix $A$ with the same Schmidt-vectors (singular vectors) as $B$, by von-Neumann's inequality \cite{mirsky1975trace}. More precisely, denote the singular values of $A,B$ by $\alpha,\beta$ and denote the singular value decomposition by $A=U_A\Sigma_\alpha V_A^*$, where $\Sigma_\alpha$ is a diagonal matrix of length $N=\min(m,n)$. We then have:

\begin{proposition}\label{vonNeumann}
For any $A,B\in \m_{m,n}$ we have $\langle{A,B}\rangle\leq \sum_{j=1}^{N} \alpha_j\beta_j$ with equality if and only if the singular vectors can be chosen such that $U_A=U_B$ and $V_A=V_B$.
\end{proposition}
See \cite{de1994exposed} for a discussion regarding the proof and the original formulation of von Neumann. To simplify the presentation, in what follows we shall occasionally write $\r_K(\alpha)$ in place of $\r_K(\Sigma_{\alpha})$ when it is suitable.

\begin{proposition}\label{p1} Let $\I$ be as defined by (\ref{fixrank_constrained_r}). For its Fenchel conjugate it then holds that
\begin{align*}&\I^*(B)= \sum_{j=1}^K\left(\sigma_j\left(F+B/2\right)\right)^2-\|F\|^2.\end{align*}
\end{proposition}
\begin{proof}
We abbreviate $\sigma_j\left(F+\frac{B}{2}\right)=\gamma_j$. Then
\begin{align*}
&\I^*(B)=\sup_A \langle A,B\rangle-\r_K(A)-\|A-F\|^2=
&\sup_\alpha -\r_K(\alpha)-\sum_{j=1}^N(\alpha_j-\gamma_j)^2+\sum_{j=1}^N\gamma_j^2-\|F\|^2.
\end{align*}
It is clear that it is optimal to choose $\alpha_j=\gamma_j$ for $1\leq j\leq K$ and $\alpha_j=0$ otherwise. Hence,
\begin{align*}&\I^*(B)= -\sum_{j=K+1}^N\gamma_j^2+\sum_{j=1}^N\gamma_j^2-\|F\|^2.  \qed \end{align*}
\end{proof}

For the computation of $\I^{**}$ some auxiliary results are needed.
\begin{lemma}\label{l1}
Let $(r_j)_{j=1}^K$ be an increasing sequence, $c\geq 0$, and set $$s_n=\frac{c+\sum_{j=1}^n r_j}{n}.$$ There exists a $1\leq k_*\leq K$ such that this is a decreasing sequence of $1\leq n\leq k_*$ and a strictly increasing sequence of $k_*\leq n\leq K$. Defining $r_{K+1}=\infty$ we have that $k_*$ is the smallest value of $n$ such that $s_n< r_{n+1}$ holds, and the largest value of $n$ such that $r_{n}\leq s_{n}$ holds. In particular, it is the unique value satisfying \begin{equation}\label{1}r_{k_*}\leq s_{k_*}< r_{{k_*}+1}.\end{equation}
\end{lemma}
\begin{proof}
Let $k_*$ be the first (i.e. smallest) value of $n$ such that \begin{equation}\label{3}s_n< r_{n+1}.\end{equation}
Note that \begin{equation}\label{average}s_{n+1}=\frac{1}{n+1}r_{n+1}+\frac{n}{n+1}s_n\end{equation} which is an weighted average,
so $s_{n+1}$ lies between $r_{n+1}$ and $s_n$. As long as \eqref{3} fails we thus have \begin{equation}\label{4} r_{n+1}\leq s_{n+1}\leq s_n,\end{equation} and reversely
\begin{equation}\label{5} r_{n+1}> s_{n+1}> s_n,\end{equation} when \eqref{3} holds. By \eqref{5} and the fact that $r_{n+2}\geq r_{n+1}$, we see that once \eqref{3} is fulfilled for some $n$, it is fulfilled for all subsequent $n$. This combined with the inequalities \eqref{4}, \eqref{5} proves the first part of the statement.

For the remaining statements, note that we have already chosen $k_*$ as the smallest value for which $s_n< r_{n+1}$ holds. Moreover, before this point (i.e. $n\leq k_*$) we do have $r_{n}\leq s_{n}$ by \eqref{4} (or, in case, $n=1$, by the definition of $s_1$) and after this point (i.e. $n\geq k_*$) we do not have it, by \eqref{5} and the fact that this holds for all $n\geq k_*$, as noted earlier. Hence $k_*$ is the largest value of $n$ such that $r_{n}\leq s_{n}$ holds. The inequality \eqref{1} and its uniqueness now immediately follows. \qed
\end{proof}
For easier reference, we reformulate the above lemma in the setting it will be used.
\begin{lemma}\label{l11}
Let $\beta\in\R^N$ be decreasing, $K<N$ fixed and set $$\omega_k=\frac{\sum_{j>K-k} \beta_j}{k}.$$ There exists a $1\leq k_*\leq K$ such that $\omega$ is a decreasing sequence of $1\leq n\leq k_*$ and a strictly increasing sequence of $k_*\leq n\leq K$. Defining $\beta_0=\infty$ we have that $k_*$ is the smallest value of $k$ such that $\omega_k< \beta_{K-k}$ holds, and the largest value of $k$ such that $\beta_{K-k+1}\leq \omega_{k}$ holds. In particular, it is the unique value satisfying \begin{equation}\label{11m}\beta_{K-k_*+1}\leq \omega_{k_*}< \beta_{K-{k_*}}.\end{equation}
\end{lemma}
\begin{proof}
Apply Lemma \ref{l1} with $c=\sum_{j>K}\beta_j$ and $r_j=\beta_{K+1-j}$.
\end{proof}
We are now ready to address \eqref{t1prel}.
\begin{theorem}\label{t1}
The convex envelope of $\r_K(A)+\|A-F\|^2$ is \begin{align*}&\I^{**}(A)= \frac{1}{k_*}\biggl(\sum_{j>K-k_*}\alpha_j\biggr)^2-\sum_{j>K-k_*}\alpha_j^2+\|A-F\|^2\end{align*} where $k_*=k_*(\alpha)$ ($1\leq k_*\leq K$) is
obtained by applying Lemma \ref{l11} with $\beta=\alpha$.
\end{theorem}

Note that $\frac{1}{k_*}\left(\sum_{j>K-k_*}\alpha_j\right)^2=k_*\omega_{k_*}^2$ in the terminology of Lemma \ref{l11}. Also note that $\I(A) \geq \|A-F\|^2$ and $\|A-F\|^2$ is convex in $A$.
Since $\I^{**}(A)$ is the largest convex lower bound on $\I(A)$ we therefore have
$\I^{**}(A) \geq \|A-F\|^2$ which shows that $\frac{1}{k_*}\left(\sum_{j>K-k_*}\alpha_j\right)^2-\sum_{j>K-k_*}\alpha_j^2\geq 0.$
\begin{proof}
We again employ the notation $\sigma_j\left(F+\frac{B}{2}\right)=\gamma_j$. For the bi-conjugate it then holds that
\begin{align*}&\I^{**}(A)=\sup_B \langle A,B\rangle-\sum_{j=1}^K\gamma_j^2+\|F\|^2=\sup_B 2\langle A,F+\frac{B}{2}\rangle-\sum_{j=1}^K\gamma_j^2+\|A-F\|^2-\|A\|^2\\&=\sup_\gamma 2\sum_{j=1}^N\alpha_j\gamma_j-\sum_{j=1}^K\gamma_j^2+\|A-F\|^2-\|A\|^2 \\&=\sup_\gamma 2\sum_{j=K+1}^N\alpha_j\gamma_j-\left(\sum_{j=1}^K(\gamma_j-\alpha_j)^2-\alpha_j^2\right)+\|A-F\|^2-\|A\|^2\end{align*}
where the middle identity follows by von-Neumann's trace inequality. We now hold $\gamma_K$ fixed and consider the supremum over the remaining variables. Given $0\leq k\leq K$ consider \begin{equation}\label{55}\gamma_K\in [\alpha_{K-k+1},\alpha_{K-k}]\end{equation} (where as before $\alpha_0=\infty$). It is not hard to see that the maximal value over the other variables is achieved by setting $\gamma_j=\gamma_K$ for $j> K-k$, and $\gamma_j=\alpha_j$ for the remaining ones.
This gives
\begin{align}&\nonumber
\sup_{\{\gamma_j,~j\neq K\}} 2\sum_{j=K+1}^N\alpha_j\gamma_j-\biggl(\sum_{j=1}^K(\gamma_j-\alpha_j)^2-\alpha_j^2\biggr)\\
&\nonumber=\sup_{\{\gamma_j,~j\neq K\}} 2\sum_{j=K+1}^N\alpha_j\gamma_K-\biggl(\sum_{j=K-k+1}^K(\gamma_K-\alpha_j)^2-\alpha_j^2\biggr)+\biggl(\sum_{j=1}^{K-k}\alpha_j^2\biggr)
\\&\label{58}=2\gamma_K\sum_{j=K-k+1}^N\alpha_j-k\gamma_K^2+\biggl(\sum_{j=1}^{K-k}\alpha_j^2\biggr):=f(\gamma_K),
\end{align}
Since $f$ by definition is defined as a partial supremum over a concave function, it follows that $f$ itself is concave (see e.g. Section 3.2.5 in \cite{boyd-vandenberghe-book-2004}). In particular, the different expressions valid in the different regimes \eqref{55} agree at overlapping endpoints. Also, the expression for $k=0$ is valid in $[0,\alpha_K]$, and since this is linear non-decreasing, the supremum of $f$ is attained in one of the other intervals (possibly at $\alpha_K$). We may thus assume that the supremum is attained in a (non-void) interval of the form \begin{equation}\label{56}\gamma_K\in [\alpha_{K-k+1},\alpha_{K-k})\end{equation}
with $k\geq 1$. By fixing $k$ and differentiating the expression for $f$ in \eqref{58}, we conclude that the maximum is obtained at the point $$\omega_k=\frac{\sum_{j=K-k+1}^N\alpha_j}{k}$$ which then must lie in the interval \eqref{56}. With $\beta=\alpha$, this constraint is precisely the inequalities \eqref{11m}, and hence appropriate $k$ equals $k_*$ given by Lemma \ref{l11} applied to $\alpha$. Moreover, by \eqref{58} we then get
\begin{align*}&\sup_{\gamma} 2\sum_{j=K+1}^N\alpha_j\gamma_j-\left(\sum_{j=1}^K(\gamma_j-\alpha_j)^2-\alpha_j^2\right)=\sup_{\gamma_K}f(\gamma_K)=f(\omega_{k_*})
\\&=2\omega_{k_*}\sum_{j=K-k_*+1}^N\alpha_j-k_*\omega_{k_*}^2+\left(\sum_{j=1}^{K-k_*}\alpha_j^2\right)=k_*\omega_{k_*}^2+\left(\sum_{j=1}^{K-k_*}\alpha_j^2\right).\end{align*}
Returning to the initial calculation we thus see that \begin{align*}&\I^{**}(A)=k_*\omega_{k_*}^2+\left(\sum_{j=1}^{K-k_*}\alpha_j^2\right)+\|A-F\|^2-\|A\|^2\end{align*}
which proves the theorem since $\|A\|^2=\sum_{j=1}^N\alpha_j^2$. \qed
\end{proof}

\section{Global minimizers}\label{secglobal}

We now consider global minimizers of $\I$ and $\I^{**}$. Given a sequence $(\phi_n)_{n=1}^N$ we recall that $\Sigma_\phi$ denotes the corresponding diagonal matrix. We introduce the notation $\tilde\phi$ for the sequence $\phi$ truncated at $K$, i.e.
\begin{equation}\label{ey}
\tilde\phi_j =
\begin{dcases}
                     \phi_j  & \mbox{if~} 1\leq j\leq K,\\
                     0 & \mbox{otherwise.}
\end{dcases}
\end{equation}
Since $K$ is a fixed number which is clear from the context, we will usually abbreviate $\tilde\phi$ by $\tilde{\phi}$.
Recall the Eckart-Young-Schmidt theorem, which can be rephrased as follows;

\textit{The solutions to $\argmin{A} \I(A)$ are all matrices of the form $A_*=U\Sigma_{\tilde\phi} V^*$, where $U\Sigma_\phi V^*$ is any singular value decomposition of $F$. $A_*$ is unique if and only if the singular value $\phi_K$ has multiplicity one.}

Obviously, a global minimizer of $\I$ is a global minimizer of $\I^{**}$, but the converse need not be true. It is not hard to see that, in case $\phi_K$ has multiplicity one, the minimizer of $\I$ is also the (unique) minimizer of $\I^{**}$. The general situation is more complicated. Given integers $m$ and $M$ with $m\leq M$, denote by $\Omega_{M,m}$ the simplex in $\R^M$ given by the hyperplane \begin{equation}\label{144}\sum_{j=1}^Mx_j=m\end{equation} and the inequalities \begin{equation}\label{14} 1\geq x_1\geq x_2\geq\ldots\geq x_M\geq 0.\end{equation}

\begin{theorem}\label{thm_multiplicity}
Let $K\in \N$ be given, let $F$ be a fixed matrix and let $\phi$ be its singular values. Let $\phi_J$ (respectively $\phi_L$) be the first (respectively last) singular value that equals $\phi_K$, and set $M=L+1-J$ (that is, the multiplicity of $\phi_K$). Finally set $m=K+1-J$, (that is, the multiplicity of $\tilde\phi_K$).

The global minimum of $\I$ and $\I^{**}$ both equal $\sum_{j>K}\phi_j^2$ and the solutions to \begin{equation}\label{6}\argmin{A} \I^{**}(A)\end{equation} are all matrices of the form $A_*=U\Sigma_{\alpha} V^*$, where $U\Sigma_\phi V^*$ is any singular value decomposition of $F$, and $\alpha$ is a decreasing sequence satisfying:
\begin{equation}\label{15}\begin{dcases}
                     \alpha_j=\phi_j, & 1\leq j< J,\\
                     (\alpha_j)_{j=J}^L\in \phi_K \Omega_{M,m}, &~\\
                     \alpha_j=0, & j>L.
                   \end{dcases}
\end{equation}
In particular, $A_*$ is unique if and only if $\phi_K$ has multiplicity one. Also, the maximal rank of such an $A_*$ is $L$ and the minimal rank is $J$.
\end{theorem}

\begin{proof}
The fact that the minimum value of $\I$ and $\I^{**}$ coincide follows immediately since $\I^{**}$ is the convex envelope of $\I$, and the fact that this value is $\sum_{j>K}\phi_j^2$ follows by the Eckart-Young-Schmidt theorem.

Let $A$ be a solution to \eqref{6}. By Proposition \ref{vonNeumann} it then follows that we can choose matrices $U$ and $V$ such that $A=U\Sigma_\alpha V^*$ and $F=U\Sigma_\phi V^*$ are singular value decompositions of $A$ and $F$ respectively. Set $\tilde F=U\Sigma_{\tilde\phi}V^*$. Note that $\tilde F$ also is a minimizer of \eqref{6}, which follows by the first sentence of the proof and the fact that $\I(\tilde{F})=\sum_{j>K}\phi_j^2$. Since $\I^{**}$ is the convex envelope of $\I$, it follows that all matrices $$A(t)=\tilde{F}+t(A-\tilde{F}), \quad 0\leq t\leq 1,$$
are solutions of \eqref{6}. Set \begin{equation}\label{999}\epsilon=\alpha-\tilde\phi,\end{equation}  where $\alpha$ are the singular values of $A$ and note that $A(t)=U\Sigma_{\tilde\phi+t\epsilon}V^*$. Since $\alpha$ is a decreasing non-negative sequence, we also get certain restrictions on $\epsilon$ such as $(\epsilon_j)_{j=J}^K$ being decreasing and $(\epsilon_j)_{j>K}$ being decreasing and non-negative.

We now compute $\I^{**}(A(t))$ according to Theorem \ref{t1} for some fixed value of $t$. Set $$\alpha(t)=\tilde\phi+t\epsilon$$ and note that the condition for choosing $k=k_*(\alpha(t))$ is \begin{equation}\label{7}\alpha_{K+1-k}(t)\leq \frac{\sum_{K+1-k}^N\alpha_j(t)}{k}<\alpha_{K-k}(t),\end{equation} (where we abbreviate $\sum_{j=K+1-k}^N$ by $\sum_{K+1-k}^N$ for simpler reading).
For small values of $t$ and $k\leq m$, all numbers above are close to $\phi_K$, except $\alpha_{K-k}(t)$ for the value $k=m$, in which case $\alpha_{K-k}(t)\approx \phi_{J-1}$ which by choice of $J$ is strictly larger than $\phi_K$. It follows by Lemma \ref{l11} that $k_*(\alpha(t))\leq m$ for values of $t$ near 0. We only consider such values in what follows. By Lemma \ref{l11} we also have that $k_*(\alpha(t))$ equals the largest value for which
\begin{equation}\label{8}\alpha_{K+1-k}(t)\leq \frac{\sum_{K+1-k}^N\alpha_j(t)}{k}\end{equation}
holds. Since $\alpha_j(t)=\phi_K+t\epsilon_j$ for $J\leq j\leq K$,
it easily follows that \eqref{8} is equivalent with
\begin{equation}\label{9}
\epsilon_{K+1-k}\leq \frac{\sum_{K+1-k}^N\epsilon_j}{k},
\end{equation} which is independent of $t$, and hence so is $k_*(\alpha(t))$. In the remainder we simply write $k_*$. It follows that 
\begin{align*}&\I^{**}(A(t))=k_*\biggl(\phi_K+t\frac{\sum_{K+1-k_*}^N \epsilon_j}{k_*}\biggr)^2-\sum_{K+1-k_*}^K(\phi_K+t\epsilon_j)^2-\sum_{K+1}^N (t\epsilon_j)^2+\|A(t)-F\|^2\\&=k_*\biggl(\phi_K+t\frac{\sum_{K+1-k_*}^N \epsilon_j}{k_*}\biggr)^2-\sum_{K+1-k_*}^K(\phi_K+t\epsilon_j)^2-\sum_{K+1}^N (t\epsilon_j)^2+\sum_{1}^K(t\epsilon_j)^2+\sum_{K+1}^N(t\epsilon_j-\phi_j)^2,\end{align*} which looks like a second degree polynomial in $t$ (with constant term $\sum_{j>K}\phi_j^2$ as it should). However, note that this polynomial is in fact constant by our assumption on $A(t)$, and hence the first and second coefficient are zero. The coefficient of the linear term is \begin{equation}\label{10}2\phi_K{\sum_{K+1-k_*}^N \epsilon_j}-2\phi_K{\sum_{K+1-k_*}^K \epsilon_j}-2\sum_{K+1}^N\phi_j\epsilon_j=2\sum_{K+1}^N(\phi_K-\phi_j)\epsilon_j.\end{equation}
Note that $\phi_K-\phi_j=0$ for $K<j\leq L$, that $\phi_K-\phi_j>0$ for $j>L$ and that $\epsilon_j$ is non-negative in this range. We conclude that \begin{equation}\label{11}\epsilon_j=0,\quad j>L.\end{equation} The coefficient of the quadratic term is
$$k_*\left(\frac{\sum_{K+1-k_*}^N \epsilon_j}{k_*}\right)^2-\sum_{K+1-k_*}^N \epsilon_j^2+\sum_{1}^N\epsilon_j^2=
k_*\left(\frac{\sum_{K+1-k_*}^L \epsilon_j}{k_*}\right)^2+\sum_{1}^{K-k_*} \epsilon_j^2$$ where we have used \eqref{11}. It clearly follows that 
\begin{equation}\label{13}
\sum_{K+1-k_*}^L \epsilon_j=0, \quad \epsilon_j=0,\quad 1\leq j\leq K-k_*.
\end{equation} 

\begin{figure}
\centering
\includegraphics[width=1\linewidth,trim=5cm 0cm 4cm 0cm]{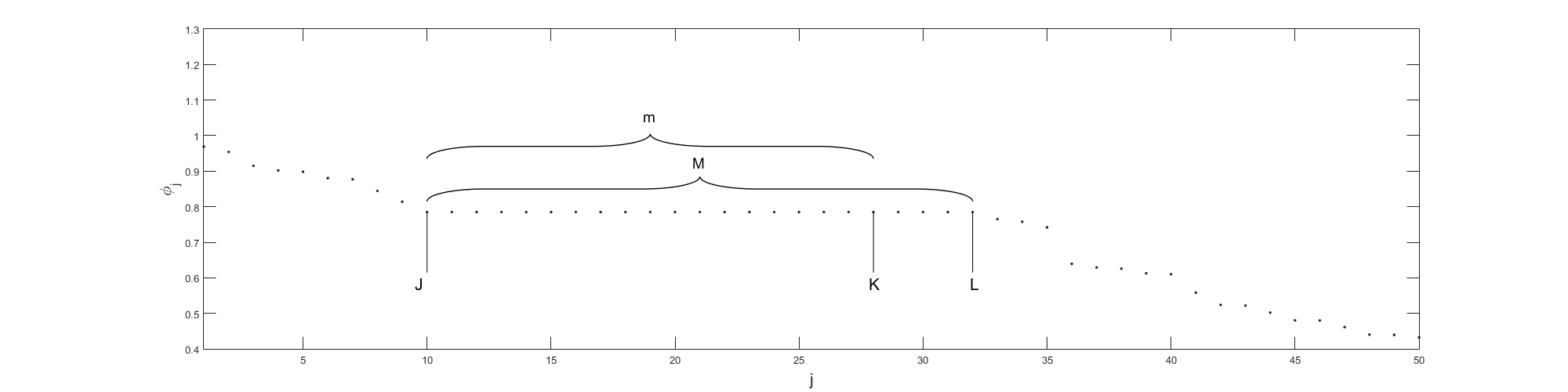}
\caption{Illustration of the notation used in Theorem \ref{thm_multiplicity}}
\end{figure}

If $k_*$ is not maximal, i.e. equal to $m$, then $\epsilon_J$ is among the $\epsilon_j$'s in \eqref{13}, which forces $(\epsilon_j)_{J}^K$ to be non-positive since this is a decreasing sequence, as noted following \eqref{999}. If $k_*$ is maximal then $K+1-k_*=J$ and \eqref{9}, \eqref{13} implies that \begin{equation}\label{99}\epsilon_{J}\leq \frac{\sum_{J}^L\epsilon_j}{k_*}=0,\end{equation}
from which the same conclusion follows. Summing up, the above, equations \eqref{11}, \eqref{13}, and the remarks following \eqref{999}, together imply that $(\epsilon_j)_{1}^N$ is zero except possibly in the interval $\{J,\ldots,K\}$, where it is decreasing non-positive, and the interval $\{K+1,\ldots,L\}$, where it is decreasing non-negative.

The entire analysis has been valid for ``small'' $t$, since at the outset we could not prove that $k_*\leq m$ for all values of $t$. By now we know that $\alpha_{J-1}(t)=\phi_{J-1}$ and that $\alpha_J(t)$ is decreasing, by which it follows that the whole previous analysis is valid in the entire range $t\in(0,1]$. The top affirmation of \eqref{15} now follows by \eqref{13}, and the bottom follows by \eqref{11}. It remains to prove the middle. Clearly $(\alpha_j(0))_{J}^L$ lies in the hyperplane \eqref{144} since the first $m$ values equal $\phi_K$ and the remaining ones are 0. From \eqref{13} it is clear that $\sum_{J}^L \epsilon_j=0,$ so we stay in this hyperplane as $t$ varies. Concerning \eqref{14}, that $(\alpha_j(t))_{J}^L$ has to be increasing is immediate by construction, and $\alpha_{J}(t)\leq 1\cdot\phi_K$ follows as $\epsilon_J\leq 0$, as noted earlier. We conclude that if $A$ is a solution to \eqref{6}, then \eqref{15} is satisfied.

Conversely, if it has this form then the calculations of the proof clearly shows that $\I^{**}(A(t))$ is a constant polynomial equal to $\sum_{j>K}\phi_j^2$. Finally, the uniqueness statement and the rank statements are immediate. \qed
\end{proof}

\section{The proximal operator}\label{prox}

\begin{theorem}\label{t3}
Let $F=U_F\Sigma_\phi V_F^*$ be given. The solution of 
\begin{equation}\label{66}
\argmin{A} \I^{**}(A)+\rho\|A-F\|^2,
\end{equation}
 is of the form $A=U_F\Sigma_\alpha V_F^*$ where $\alpha$ has the following structure; there exists natural numbers $k_1 \leq K\leq k_2 $ and real number $s>\phi_{k_2}$ such that \begin{equation}\label{4tg}\alpha_j=\left\{
                            \begin{array}{ll}
                              \phi_j, &\quad j< k_1 \\
                              \phi_j-\frac{s-\phi_j}{\rho}, &\quad k_1\leq j\leq k_2\\
                              0,&\quad j>k_2
                            \end{array}
\right.\end{equation}
In particular, $\alpha$ is a decreasing sequence and $\alpha\leq \phi$. In other words, the proximal operator is a contraction.
\end{theorem}

The theorem can be deduced by working directly with the expression for $\I^{**}$, but it turns out that it is easier to follow the approach in \cite{larsson2015convex} which is based on the minimax theorem and an analysis of the simpler functional $\I^*$. We give more concrete information about how to find $s$, $k_1$ and $k_2$ in a separate proposition after the proof.

\begin{proof}
The proof partially follows the approach in \cite{larsson2015convex}, Section 3.1. Using Proposition \ref{p1} and some algebraic simplifications, \eqref{66} can be rewritten
\begin{align*}&\label{663}\argmin{A} \max_B \langle A,B\rangle-\I^{*}(B)+\rho\|A-F\|^2=\\&\argmin{A} \max_B \langle A,B\rangle-\sum_{j=1}^K\left(\sigma_j\left(F+\frac{B}{2}\right)\right)^2+\|F\|^2+\rho\|A-F\|^2\\&\argmin{A} \max_Z \rho\left\|A-\frac{(1+\rho)F-Z}{\rho}\right\|^2-\frac{1}{\rho}\|Z-{(1+\rho)F}\|^2+(1+\rho)\|F\|^2-\sum_{j=1}^K\left(\sigma_j(Z)\right)^2.\end{align*}
Switching the order of $\max$ and $\min$ gives the relation $A=((1+\rho)F-Z)/\rho,$ and this in turn yields that the maximization over $Z$ takes the form
\begin{align*}
\argmax{Z} -\frac{1}{\rho}\|Z-{(1+\rho)F}\|^2-\sum_{j=1}^K\zeta_j^2, \quad \mbox{~where~} \zeta_j=\sigma_j(Z).
\end{align*}
By Proposition \eqref{vonNeumann} it follows that the appropriate $Z$ shares singular vectors with $F$, so the problem reduces to that of minimizing
\begin{align*}
\argmin{\zeta} \sum_{j=1}^N(\zeta_j-(1+\rho)\phi_j)^2+\rho\sum_{j=1}^K\zeta_j^2=\argmin{\zeta} (1+\rho)\sum_{j=1}^K(\zeta_j-\phi_j)^2+\sum_{j=K+1}^N(\zeta_j-(1+\rho)\phi_j)^2.
\end{align*}
The unconstrained minimization (i.e. ignoring that the singular values need to be decreasing) of this is $\zeta_j=\phi_j$ for $j\leq K$ and $\zeta_j=(1+\rho)\phi_j$ for $j>K$.
It is proven in the appendix of \cite{larsson2015convex} that the constrained minimization has the solution
\begin{equation}\label{4tgg}\zeta_j=\left\{
                            \begin{array}{ll}
                              \max(\phi_j,s), &\quad j\leq K \\
                              \min((1+\rho)\phi_j,s),&\quad j>K
                            \end{array}
\right.\end{equation}
where $s$ is a parameter between $\phi_K$ and $(1+\rho)\phi_{K+1}$. The appropriate value of $s$ is easily found by inserting this into the expression above. Let $k_1$ resp. $k_2$ be the first resp. last index where $s$ shows up in $\zeta$. Formula \eqref{4tg} is now an easy consequence of \eqref{4tgg}.
\end{proof}


\begin{proposition} The appropriate value of $s$ is found by minimizing
\begin{align*} \sum_{j=1}^K\left(\max(\phi_j,s)-\phi_j\right)^2+\sum_{j=K+1}^N\left(\min(\phi_j,\frac{s}{1+\rho})-\phi_j\right)^2.\end{align*}
in the interval $[\phi_K, (1+\rho)\phi_{K+1}]$. Given such an $s$, $k_1$ is the smallest index $\phi$ with $\phi_{k_1}<s$ and $k_2$ last index with $\phi_{k_2}>\frac{s}{1+\rho}$.
\end{proposition}

%

Note in particular that the proximal operator (given by Theorem \ref{t3}) reduce to  \eqref{ey} if 

\begin{equation}\label{n}\phi_K\geq (1+\rho)\phi_{K+1}.\end{equation}

\section{Conclusions}
We have analyzed and derived expressions for how to compute the convex envelope corresponding to the problem of finding the best approximation to a given matrix with a prescribed rank. These expressions work directly on the singular values. 
\section{Acknowledgements}
This research is partially supported by the Swedish Research Council, grants no. 2011-5589, 2012-4213 and 2015-03780; and the Crafoord Foundation.

\bibliographystyle{plain}
\bibliography{fixed_rank_envelope}
\end{document}